\documentclass[12pt]{article}

\usepackage{amsmath}
\usepackage{amssymb}
\usepackage{amsthm}
\usepackage{enumerate}
\usepackage{cases}
\usepackage{stmaryrd}
\usepackage{color}
\usepackage[all]{xy}

\usepackage{hyperref}

\usepackage[top=1 in, bottom=1 in, left=1 in, right=1 in]{geometry}

\numberwithin{equation}{section} 
\newcommand\numberthis{\stepcounter{equation}\tag{\theequation}} 

\newcounter{constantno}
\newcommand{\constantnumber}[1]{\refstepcounter{constantno}\label{#1}}

\newtheorem{lemma}{Lemma}[section]
\newtheorem{theorem}[lemma]{Theorem}

\newtheorem{corollary}[lemma]{Corollary}

\theoremstyle{definition}

\theoremstyle{remark}

\newtheorem{remark}[lemma]{Remark}

\newcommand{\ba}{\bar{\alpha}}  
\newcommand{\bb}{\bar{\beta}}

\newcommand{\bl}{\bar{\lambda}}
\newcommand{\bm}{\bar{\mu}}

\newcommand{\la}{\langle}
\newcommand{\ra}{\rangle}

\begin{document}
\title{Horizontal Gradient Estimate of Positive Pseudo-Harmonic Functions on Complete Noncompact Pseudo-Hermitian Manifolds
}

\author{Yibin Ren}

\maketitle

\begin{abstract}
	In this paper, we will give a horizontal gradient estimate of positive solutions of $\Delta_b u = - \lambda u$ on complete noncompact pseudo-Hermitian manifolds. As a consequence, we recapture the Liouville theorem of positive pseudo-harmonic functions on Sasakian manifolds with nonnegative pseudo-Hermitian Ricci curvature.
\end{abstract}

\section{Introduction}

Let $(M, \theta)$ be a pseudo-Hermitian manifold and $\nabla$ be the Tanaka-Webster connection.
The sub-Laplacian of a smooth function $u$ is defined by 
\begin{align*}
	\Delta_b u = \mbox{trace}_{G_\theta} \nabla_b d_b u
\end{align*}
where $\nabla_b d_b u$ is the restriction of $\nabla d u$ on $HM \times HM$.
It is a subelliptic operator in pseudo-Hermitian geometry and its local theories are close to elliptic operator (cf. \cite{folland1974estimates,jost1998subelliptic,rothschild1976hypoelliptic,strichartz1986sub}).
A smooth function is called pseudo-harmonic if its sub-Laplacian vanishes.

Cheng and Yau derived a well-known gradient estimate of positive harmonic functions on Riemannian manifolds.
\begin{theorem}[\cite{cheng1975differential,yau1975harmonic}]
	Suppose that $M$ is a $m$-dimensional complete manifold with $Ric \geq - (m-1) \kappa$ for some $\kappa \geq 0$. If $u$ is a positive harmonic function on a geodesic ball $B_{2R} (x_0)$, then 
	\begin{align*}
		\sup_{B_R (x_0)} \frac{|\nabla u|}{u} \leq C_m \left( \frac{1}{R} + \sqrt{\kappa} \right)
	\end{align*}
	where $C_m$ is a constant depending on $m$.
\end{theorem}
As an consequence, Liouville theorem holds for positive harmonic functions on Riemannian manifolds with nonnegative Ricci curvature. 
This method is also very important for geometric and analytic objects, such as eigenvalues, eigenfunctions, heat kernel and Harnack inequality.
One can refer to \cite{li2012geometric} for these discussions.
Due to the similarity of harmonic functions, it is natural to consider the Liouville theorem for pseudo-harmonic functions on pseudo-Hermitian manifolds. 
The key point of such generalization is the analogue of Laplacian comparison theorem in pseudo-Hermitian manifolds.
The authors in \cite{Agrachev2015bishop,baudoin2017comparison,lee2013bishop} have studied the sub-Laplacian comparison theorem and Hessian comparison theorem of sub-Riemannian distance function (also called Carnot-Carath\'eodory distance) on Sasakian manifolds.
This was also explored by Chang, Kuo, Lin and Tie \cite{chang2015gradient} via a different method. 
They applied it to generalize Cheng-Yau's theorem and established the Liouville theorem for positive pseudo-harmonic functions on Sasakian manifolds with nonnegative pseudo-Hermitian Ricci curvature.

It is notable that Webster metric is a Riemannian metric which gives a Riemannian distance function.
The authors in \cite{chong2017harmonicnew} have estimated its sub-Laplacian on pseudo-Hermitian manifolds which is a weak version of sub-Laplacian comparison theorem, and deduced some Liouville theorem and existence theorem of pseudo-harmonic maps to regular balls of some Riemannian manifolds which is a generalization of harmonic case \cite{cheng1980liouville,choi1982liouville}.
Motivated by these, this paper will establish the horizontal gradient estimate of positive eigenfunctions of sub-Laplacian and estimate the upper bound of eigenvalues on complete noncompact pseudo-Hermitian manifolds in Theorem \ref{c-thm-general}.
Meanwhile, the result will be simplified for Sasakian manifolds in Corollary \ref{c-thm-sasakian}.
As an consequence, the Liouville theorem for positive pseudo-harmonic functions on Sasakian manifolds will be recaptured in Theorem \ref{c-thm-liouville}.

\section{Preliminary}

In this section, we present some basic notions of pseudo-Hermitian geometry. For details, the readers may refer to \cite{dragomir2006differential,webster1978pseudo}. Recall that a smooth manifold $M$ of real dimension ($2n+1$) is said to be a CR manifold if
there exists a smooth rank $n$ complex subbundle $T_{1,0} M \subset TM \otimes \mathbb{C}$ such that
\begin{gather}
T_{1,0} M \cap T_{0,1} M =0 \\
[\Gamma (T_{1,0} M), \Gamma (T_{1,0} M)] \subset \Gamma (T_{1,0} M) \label{b-integrable}
\end{gather}
where $T_{0,1} M = \overline{T_{1,0} M}$ is the complex conjugate of $T_{1,0} M$.
Equivalently, the CR structure may also be described by the real subbundle $HM = Re \: \{ T_{1,0} M \oplus T_{0,1}M \}$ of $TM$, called the horizontal bundle, which carries a almost complex structure $J : HM \rightarrow HM$ defined by $J (X+\overline{X})= i (X-\overline{X})$ for any $X \in T_{1,0} M$.
Since $HM$ is naturally oriented by the almost complex structure $J$, then $M$ is orientable if and only if
there exists a global nowhere vanishing 1-form $\theta$ such that $ HM = Ker (\theta) $.
Any such section $\theta$ is referred to as a pseudo-Hermitian structure on $M$. The space of all pseudo-Hermitian structure is 1-dimensional The Levi form $L_\theta $ of a given pseudo-Hermitian structure $\theta$ is defined by
$$L_\theta (X, Y ) = d \theta (X, J Y)  $$
for any $X, Y \in HM$. 
An orientable CR manifold $(M, HM, J)$ is called strictly pseudo-convex if $L_\theta$ is positive definite for some $\theta$.

When $(M, HM, J)$ is strictly pseudo-convex, there exists a pseudo-Hermitian structure $\theta$ such that $L_\theta$ is positive. The quadruple $( M, HM, J, \theta )$ is called a pseudo-Hermitian manifold. 
For simplicity, we denote it by $(M, \theta)$.
This paper is discussed in these pseudo-Hermitian manifolds.

For a pseudo-Hermitian manifold $(M, \theta)$, there exists a unique nowhere zero vector field $\xi$, called the Reeb vector field, transverse to $HM$ and satisfying
$\xi \lrcorner \: \theta =1, \ \xi \lrcorner \: d \theta =0$. 
There is a decomposition of the tangent bundle $TM$: 
\begin{align}
TM = HM \oplus \mathbb{R} \xi 
\end{align}
which induces the projection $\pi_H : TM \to HM$. Set $G_\theta = \pi_H^* L_\theta$. Since $L_\theta$ is a metric on $HM$, it is natural to define a Riemannian metric
\begin{align}
g_\theta = G_\theta + \theta \otimes \theta
\end{align}
which makes $HM$ and $\mathbb{R} \xi$ orthogonal. Such metric $g_\theta$ is called Webster metric, also denoted by $\la \cdot , \cdot \ra$.
By requiring that $J \xi=0$, the complex structure $J$ can be extended to an endomorphism of $TM$.
The integrable condition \eqref{b-integrable} guarantees that $g_\theta$ is $J$-invariant.

It is remarkable that $(M, HM, G_\theta)$ could also be viewed as a sub-Riemannian manifold which satisfies strong bracket generating hypothesis. The completeness is well settled under the Carnot-Carath\'eorody distance (cf. \cite{strichartz1986sub}). By definition, this distance is larger than Riemannian distance associated with the Webster metric $g_\theta$ which implies that sub-Riemannian completeness is stronger than Riemannian one. 
In this paper, a pseudo-Hermitian manifold $(M, \theta)$ is called complete if it is complete as a Riemannian manifold $(M, g_\theta)$.

On a pseudo-Hermitian manifold, there exists a canonical connection $\nabla$ preserving the horizontal bundle, the CR structure and the Webster metric. Moreover, its torsion satisfies
\begin{gather*}
T_{\nabla} (X, Y)= 2 d \theta (X, Y) \xi  \mbox{ and }T_{\nabla} (\xi, J X) + J T_{\nabla} (\xi, X) =0.
\end{gather*}

The pseudo-Hermitian torsion, denoted by $\tau$, is a symmetric and traceless tensor defined by $\tau (X) = T_\nabla (\xi, X)$ for any $X \in TM$ (cf. \cite{dragomir2006differential}).
Set $A = g_\theta ( \tau (X) , Y)$ for any $X, Y \in TM$.
A pseudo-Hermitian manifold is called Sasakian if $\tau \equiv 0$.
Sasakian geometry is very rich as the odd-dimensional analogous of K\"ahler geometry. We refer the readers to the book \cite{boyer2008sasakian} by Boyer and Galicki.

Let $R$ be the curvature tensor of the Tanaka-Webster connection. As the Riemannian curvature, $R$ satisfies
\begin{align*}
\langle R(X, Y) Z, W \rangle = - \langle R(X, Y) W, Z \rangle = - \langle R(Y, X) Z, W \rangle
\end{align*}
for any $X, Y, Z, W \in \Gamma (TM)$. 
Assume that $\{ \eta_\alpha \}_{\alpha=1}^n$ is a local unitary frame of $T_{1,0} M$.
Besides $R_{\bar{\alpha} \beta \lambda \bar{\mu}} = \langle R(\eta_\lambda, \eta_{\bar{\mu}}) \eta_\beta, \eta_{\bar{\alpha}} \rangle$, the other parts of $R$ are clear (cf. \cite{dragomir2006differential,webster1978pseudo}):
\begin{equation} \label{b-cur-tor}
	\begin{gathered}
	R_{\ba \beta \bl \bm}  = 2 i ( A_{\ba \bm} \delta_{\beta \bl} - A_{ \ba \bl} \delta_{\beta \bm} ) , \\
	R_{\ba \beta \lambda \mu} =2 i ( A_{\beta \mu} \delta_{\ba \lambda}- A_{\beta \lambda } \delta_{\ba \mu } ) , \\
	R_{\ba \beta 0 \bm} = A_{\ba \bm, \beta} , \\
	R_{ \ba \beta 0  \mu} = - A_{\mu \beta, \ba } , 
	\end{gathered}
\end{equation}
where $A_{\mu \beta, \bar{\alpha}}$ are the components of $\nabla A$. 
The pseudo-Hermitian Ricci operator $R_*$ is given by 
\begin{align}
R_* X = - i \sum_{\lambda=1}^n R(\eta_\lambda, \eta_{\bar{\lambda}}) JX.
\end{align}
Due to the first Bianchi identity $R_{\ba \beta \lambda \bm} = R_{\ba \lambda \beta \bm}$ (cf. \cite{webster1978pseudo}), the components of $R_*$ satisfies that 
$$ R_{\alpha \bar{\beta}} = R_{\bar{\beta} \alpha \lambda \bar{\lambda}} = R_{\bar{\beta} \lambda \alpha \bar{\lambda}}. $$

The sub-Laplacian of a smooth function $v$ is defined by 
\begin{align*}
	\Delta_b v = \mbox{trace}_{G_\theta} \nabla_b d_b v
\end{align*}
where $\nabla_b d_b v$ is the restriction of $\nabla d v$ on $HM \times HM$.
A smooth function is called pseudo-harmonic if its sub-Laplacian vanishes.
The CR Bochner formulas of functions were first derived by \cite{greenleaf1985first} for the estimate of first eigenvalue of sub-Laplacian on compact pseudo-Hermitian manifolds. 
The CR Bochner formulas of pseudo-harmonic maps were deduced in \cite{chang2013existence,Ren201447}.

\begin{lemma}[CR Bochner Formulas]
	Suppose that $(M^{2m+1}, \theta)$ is a pseudo-Hermitian manifold. Then any smooth function $v$ satisfies that 
	\begin{align}
		\frac{1}{2} \Delta_b |\nabla_b v|^2=& |\nabla_b d_b v|^2 + \la \nabla_b \Delta_b v, \nabla_b v \ra + 4 \langle \nabla_b v_0 , J \nabla_b v \rangle \nonumber \\
		& + \langle \left(R_* + 2 (m-2) \tau J \right) (\nabla_b v), \nabla_b v \rangle \label{b-bochner3} \\
		\frac{1}{2} \Delta_b |v_0|^2 =& |\nabla_b v_0|^2 + v_0 \cdot \nabla_\xi \Delta_b v   \nonumber \\
		& + 2 (\mbox{div } A) (\nabla_b v) \cdot v_0 + 2 \langle A, \nabla_b d_b v \rangle \cdot v_0 \label{b-bochner4}
	\end{align}
	where $v_0 = d v (\xi)$, $(\mbox{div A}) (X) = \mbox{trace}_{G_\theta} \nabla_{\bullet} A (X, \bullet)$ and $\nabla_b v$ is the horizontal gradient of $v$.
\end{lemma}

Since the proof is simple under the CR analogue of Ricci identity, that is if $  \sigma \in \Gamma ( \otimes^p T^* M )  $ and $ X_1 , \cdots , X_p \in \Gamma (TM) $, then 
\begin{align} \label{b-ricidentity}
	& ( \nabla^2 \sigma )   ( X_1 , \cdots , X_p ; X, Y ) - ( \nabla^2 \sigma )   ( X_1 , \cdots , X_p ; Y, X ) \nonumber \\
	&= \sigma  ( X_1 , \cdots , R(X, Y) X_i, \cdots,  X_p ) + \big( \nabla_{T_\nabla (X, Y) }  \sigma \big) ( X_1 , \cdots , X_p ),
\end{align}
we would recapture it here.

\begin{proof}
	Let's only prove \eqref{b-bochner3}. The equation \eqref{b-bochner4} is left to the readers. Suppose that $\{ \eta_\alpha \}_{\alpha=1}^n$ is a local unitary frame of $T_{1,0} M$. By definition, we have 
	\begin{align}
		\frac{1}{2} \Delta_b |\nabla_b v|^2 = (v_\alpha v_{\bar{\alpha}} )_{\beta \bar{\beta}} + (v_\alpha v_{\bar{\alpha}} )_{\bar{\beta} \beta } = v_{\alpha \beta \bar{\beta}} v_{\bar{\alpha}} + v_{\alpha \beta} v_{\bar{\alpha} \bar{\beta}} + v_{\alpha \bar{\beta}} v_{\bar{\alpha} \beta} + v_\alpha v_{\bar{\alpha} \beta \bar{\beta}} + \mbox{conj.} \label{b-3}
	\end{align}
	Using \eqref{b-cur-tor} and \eqref{b-ricidentity}, we have the following calculation:
	\begin{align*}
		v_{\alpha \beta \bar{\beta}} & = v_{\beta \alpha \bar{\beta}} \\
		& = v_{\beta \bar{\beta} \alpha} + R_{\alpha \bar{\gamma}} v_\gamma + 2 i  v_{\alpha 0} \\ 
		& = v_{\beta \bar{\beta} \alpha} + R_{\alpha \bar{\gamma}} v_\gamma + 2 i  v_{0 \alpha} - 2 i A_{\alpha \gamma} v_{\bar{\gamma}}, \numberthis \label{b-1}
	\end{align*}
	and 
	\begin{align*}
		v_{\bar{\alpha} \beta \bar{\beta}} & = (v_{\beta \bar{\alpha}} - 2i \delta_{\beta \bar{\alpha}} v_0)_{\bar{\beta}} = v_{\beta \bar{\alpha} \bar{\beta}} - 2 i v_{0 \bar{\alpha}} \\
		& = v_{\beta \bar{\beta} \bar{\alpha}} + R_{\bar{\gamma} \beta \bar{\alpha} \bar{\beta} } v_\gamma - 2 i v_{0 \bar{\alpha}} \\
		& = v_{\beta \bar{\beta} \bar{\alpha}} - 2i (m-1) A_{\bar{\gamma} \bar{\alpha}}  v_\gamma - 2 i v_{0 \bar{\alpha}}. \numberthis \label{b-2}
	\end{align*}
	Substituting \eqref{b-1} and \eqref{b-2} into \eqref{b-3}, we obtain \eqref{b-bochner3} with the following identities 
	\begin{align*}
		i (v_{\ba} v_{0 \alpha } - v_\alpha v_{0  \ba} ) = \langle \nabla_b v_0 , J \nabla_b v \rangle
	\end{align*}
	and
	\begin{align*}
		2 R_{\alpha \bb} v_{\ba} v_{\beta} - 2i (m-2) (v_\alpha v_\beta A_{\ba \bb}- v_{\ba} v_{\bb} A_{\alpha \beta}  ) = & \langle R_* (\nabla_b v), \nabla_b v \rangle + 2 (m-2) A (\nabla_b v, J \nabla_b v) \\
		= & \langle \left(R_* + 2 (m-2) \tau J \right) (\nabla_b v), \nabla_b v \rangle .
	\end{align*}
\end{proof}

\begin{lemma} \label{b-lem-bochner}
	Suppose that $(M^{2m+1}, \theta)$ is a pseudo-Hermitian manifold.
	\begin{align*}
	R_* + 2 (m-2) \tau J \geq -  k, \mbox{ and } |A|, | \mbox{div} A | \leq k_1,
	\end{align*}
	for some $k , k_1 \geq 0$ and $v \in C^\infty (M)$. Then for any $\epsilon \geq 0 , \epsilon_1 > 0$, 
	\begin{align}
		\Delta_b |\nabla_b v|^2 \geq &   \frac{1}{m} (\Delta_b v)^2 +  4m  |v_0|^2 + 2 |\pi_{(1,1)}^{\perp} \nabla_b d_b v|^2 \nonumber \\
		& + 2 \langle \nabla_b \Delta_b v , \nabla_b v \rangle - \epsilon_1 |\nabla_b v_0|^2 - (2 k + 16 \epsilon_1^{-1} ) |\nabla_b v|^2 \label{b-bochner1}
	\end{align}
	and
	\begin{align}
		\Delta_b |v_0|^2 \geq &  2 |\nabla_b v_0|^2 + 2 v_0 \cdot \nabla_\xi \Delta_b v  \nonumber \\
		& - 2 k_1 | \pi_{(1,1)}^{\perp} \nabla_b d_b v |^2 - 4 k_1 |v_0|^2 - 2 k_1 |\nabla_b v|^2
	\end{align}
\end{lemma}

\begin{proof}
	The CR Ricci identity \eqref{b-ricidentity} gives that 
	\begin{align*}
		v_{\alpha \bar{\beta}} - v_{\bar{\beta} \alpha} = 2 i v_0 \delta_{\alpha \bar{\beta}}
	\end{align*}
	which shows that
	\begin{align}
	| \pi_{(1,1)} \nabla_b d_b v|^2 \geq & 2 \sum_{\alpha, \beta=1}^m v_{\alpha \bar{\alpha}} v_{\bar{\alpha} \alpha} 
	= \frac{1}{2} \sum_{\alpha=1}^m \big[ |v_{\alpha \ba} +v_{\ba \alpha} |^2 + |v_{\alpha \ba} - v_{\ba \alpha} |^2  \big] \nonumber \\ 
	\geq & \frac{1}{2m} \left| \sum_{\alpha=1}^m (v_{\alpha \bar{\alpha}} + v_{\bar{\alpha} \alpha}) \right|^2 + \frac{1}{2} \sum_{\alpha=1}^m |v_{\alpha \ba} -v_{\ba \alpha} |^2  \nonumber \\
	=& \frac{1}{2m} (\Delta_b v)^2 + 2m |v_0|^2. \label{b-commutation}
	\end{align}
	The other part can be obtained by Cauchy inequality.
\end{proof}

Comparison theorem is an powerful tool in Riemannian geometry. 
There are two distance functions in pseudo-Hermitian geometry.
One is sub-Riemannian distance function (also called Carnot-Carath\'eodory distance); the other is Riemannian distance function associated with the Webster metric. 
The smoothness of the latter is better than the former (cf. \cite{strichartz1986sub}). 
The sub-Laplacian comparison theorem of the former has been derived in Sasakian manifolds (cf. \cite{Agrachev2015bishop,baudoin2017comparison,chang2015gradient,lee2013bishop});
the latter has similar sub-Laplacian estimate (cf. \cite{chong2017harmonic}) which holds for general pseudo-Hermitian manifolds.
For our purpose, we will use the Riemannian distance function to construct cut-off functions.
The following sub-Laplacian comparison theorem is based on Riemannian Index Lemma (cf. Page 212 of \cite{do1992riemannian}) and the relationship between the pseudo-Hermitian curvature $R$ and the Riemannian curvature associated with the Webster metric (cf. Theorem 1.6 in Page 49 of \cite{dragomir2006differential}).
Let $r (x)$ be the Riemannian distance function from $x\in M$ to some fixed point $x_0 \in M$.

\constantnumber{cst-subcomparison}

\begin{theorem}[\cite{chong2017harmonic}]
	Suppose $(M^{2m +1}, \theta)$ is a complete pseudo-Hermitian manifold with 
	\begin{align*}
	R_* \geq - k, \mbox{ and } |A|, | \mbox{div} A | \leq k_1
	\end{align*}
	for some $k, k_1 \geq 0$. Then there exists $C_{\ref*{cst-subcomparison}}$ only depending on $m$ such that
	\begin{align}
	\Delta_b r \leq C_{\ref*{cst-subcomparison}} \left( \frac{1}{r} + \sqrt{1+ k + k_1 + k_1^2} \right) \label{a-riemdiscprs}
	\end{align}
	inside the cut locus of $x_0$.
\end{theorem}

\section{Horizontal Gradient Estimates}

Suppose that $(M^{2m+1}, \theta)$ is a complete noncompact pseudo-Hermitian manifold.
\begin{align}
R_* + 2 (m-2) \tau J \geq -  k, \mbox{ and } |A|, | \mbox{div} A | \leq k_1, \label{c-assumption}
\end{align}
for some $k , k_1 \geq 0$. 
Let $r (x)$ be the Riemannian distance function from $x\in M$ to some fixed point $x_0 \in M$ and $B_R = B_R (x_0)$ is the Riemannian ball of radius $R$ centered at $x_0$.
Assume that $u$ is a positive solution of
\begin{align}
	\Delta_b u = - \lambda u, \mbox{ on } B_{2R}
\end{align}
for some $\lambda \geq 0$ and $R \geq 1$.
In this section, we will estimate $\dfrac{|\nabla_b u|^2}{u^2}$ on $B_R$.

Set $v = \ln u$ and then
\begin{align}
	\Delta_b v = \mbox{div } \nabla_b v = \mbox{div } \left( \frac{\nabla_b u}{u} \right) = \frac{\Delta_b u}{u} - \frac{|\nabla_b u|^2}{u^2} = - (\lambda + |\nabla_b v|^2). \label{c-1}
\end{align}
Hence 
\begin{align*}
	v_0 \cdot \nabla_\xi \Delta_b v = - v_0 \cdot \nabla_\xi |\nabla_b v|^2 = - 2 \langle v_0 \nabla_\xi \nabla_b v , \nabla_b v \rangle.
\end{align*}
For any $X \in HM$, we have 
\begin{align*}
	(\nabla_\xi v) (X) - (\nabla_X v) (\xi) = - dv (T_\nabla (\xi, X))
\end{align*}
which shows that
\begin{align*}
	\nabla_\xi \nabla_b v = \nabla_b v_0 - \tau (\nabla_b v).
\end{align*}
Hence
\begin{align*}
	v_0 \cdot \nabla_\xi \Delta_b v = -  \langle \nabla_b | v_0 |^2, \nabla_b v \rangle + 2 v_0 \cdot A (\nabla_b v, \nabla_b v).
\end{align*}
By the assumption \eqref{c-assumption} and Lemma \ref{b-lem-bochner}, we get
\begin{align}
	\Delta_b |v_0|^2 \geq &  2 |\nabla_b v_0|^2 - 2 \langle \nabla_b | v_0 |^2, \nabla_b v \rangle - 4 k_1 |\nabla_b v|^2 |v_0| \nonumber \\
	& - 2 k_1 | \pi_{(1,1)}^{\perp} \nabla_b d_b v |^2 - 4 k_1 |v_0|^2 - 2 k_1 |\nabla_b v|^2 \label{c-bochner2}
\end{align}

\constantnumber{cst-cutoff}
Choose a cut-off function $\varphi \in C^\infty ([0,\infty)) $ such that
\begin{align*}
\varphi \big|_{[0,1]} =1 , \quad  \varphi \big|_{[2, \infty)} =0,  \quad - C_{\ref*{cst-cutoff}}' | \varphi |^{\frac{1}{2}} \leq \varphi' \leq 0.
\end{align*}
By defining $ \chi (r) = \chi (\frac{r}{R})$, since $R \geq 1$, we find
\begin{align}
\frac{|\nabla_b \chi|^2}{\chi} \leq \frac{C_{\ref*{cst-cutoff}}}{R^2} , \quad \Delta_b \chi \geq - \frac{C_{\ref*{cst-cutoff}}}{R} \label{b-cutoff}
\end{align}
where $C_{\ref*{cst-cutoff}} $ depends on $ k, k_1 $. 
Set
\begin{align*}
	\Phi = |\nabla_b v|^2 + \mu_R \chi |v_0|^2
\end{align*}
where $ \mu_R $ will be determined later. 
For convenience, it will also be denoted by $\mu = \mu_R$.
But we should be careful about global estimates.

\begin{lemma} \label{c-lem-bochner}
	Suppose $k_1 \mu \leq  1$.
	If $\chi(x) \neq 0$ and $\Phi (x) \neq 0$,
	then at $x$, we have
	\begin{align} 
		\Delta_b \Phi \geq & \frac{1}{m} (\Delta_b v)^2 - 2 \langle \nabla_b \Phi, \nabla_b v \rangle + 2 \mu |v_0|^2 \langle \nabla_b \chi, \nabla_b v \rangle - 4 k_1 \mu \chi |\nabla_b v|^2 |v_0|  \nonumber \\
		& + \left(4m - 4 k_1 \mu \chi - 4 \mu \chi^{-1} |\nabla_b \chi|^2 + \mu \Delta_b \chi \right) |v_0|^2 \nonumber \\
		& -  \left[ 2 k + 2 k_1 \mu \chi + 16 ( \mu \chi)^{-1} \right] |\nabla_b v|^2 \label{b-8}
	\end{align}
\end{lemma}

\begin{proof}
	The estimate \eqref{c-bochner2} gives that
	\begin{align*}
		\Delta_b (\chi |v_0|^2) =& \chi \Delta_b |v_0|^2 + 2 \langle \nabla_b \chi, \nabla_b |v_0|^2 \rangle + |v_0|^2 \Delta_b \chi \\
		\geq & 2 \chi |\nabla_b v_0|^2 - 2 \langle \nabla_b (\chi | v_0 |^2), \nabla_b v \rangle + 2 |v_0|^2 \langle \nabla_b \chi, \nabla_b v \rangle - 4 k_1 \chi |\nabla_b v|^2 |v_0| \\ 
		& - 2 k_1 \chi | \pi_{(1,1)}^{\perp} \nabla_b d_b v |^2 - 4 k_1 \chi |v_0|^2 - 2 k_1 \chi |\nabla_b v|^2 \\ 
		& + 4 \langle \nabla_b \chi \otimes v_0, \nabla_b v_0 \rangle + |v_0|^2 \Delta_b \chi. \numberthis \label{c-reeb}
	\end{align*}
	Combing with \eqref{b-bochner1} and $\epsilon_1 = \mu \chi$, we have
	\begin{align*}
		\Delta_b \Phi =& \Delta_b (|\nabla_b v|^2 + \mu \chi |v_0|^2) \\
		\geq &  4 \mu \langle \nabla_b \chi \otimes v_0, \nabla_b v_0 \rangle +  \mu \chi  |\nabla_b v_0|^2 \\
		& + \frac{1}{m} (\Delta_b v)^2 - 2 \langle \nabla_b \Phi, \nabla_b v \rangle + 2 \mu |v_0|^2 \langle \nabla_b \chi, \nabla_b v \rangle - 4 k_1 \mu \chi |\nabla_b v|^2 |v_0| \\
		&+ 4 m |v_0|^2 - 4 k_1 \mu \chi |v_0|^2 + \mu \Delta_b \chi |v_0|^2\\
		& - \left[ 2 k + 2 k_1 \mu \chi + 16 (\epsilon \mu \chi)^{-1} \right] |\nabla_b v|^2  \numberthis \label{b-7}
	\end{align*}
	The proof is finished by Cauchy inequality
	\begin{align*}
		4 \mu \langle \nabla_b \chi \otimes v_0, \nabla_b v_0 \rangle \geq - \mu \chi |\nabla_b v_0|^2 - 4 \mu \chi^{-1} |\nabla_b \chi|^2 |v_0|^2.
	\end{align*}
	
\end{proof}

Since $\chi \Phi =0$ on $\partial B_{2 R}$, then the maximum point $x_\mu$ must lie in $B_{2 R}$.
We can assume that $x_\mu$ is inside the cut locus of $x_0$ and thus the Riemannian distance function $r$ is smooth at $x_\mu$. (Otherwise we use S.Y. Cheng's method \cite{cheng1980liouville} to modify it.) 
Set
\begin{align*}
	P = \chi |\nabla_b v|^2, Q = \chi^2 |v_0|^2, S= \chi \Phi = P + \mu Q
\end{align*}
and 
\begin{align*}
	P_\mu = P (x_\mu) , Q_\mu = Q (x_\mu) , S_\mu = S (x_\mu) .
\end{align*}

\constantnumber{cst-cutoff-2}
\begin{lemma}
	If the maximum value of $\chi \Phi$ is nonzero, that is $S_\mu \neq 0$, then for any $\mu$ with $k_1 \mu \leq 1$,
	\begin{align*}
		0 \geq & \frac{1}{m} P_\mu^2 - \left[ 2 k + 2 k_1 \mu + 16 \mu^{-1} + C_{\ref*{cst-cutoff-2}} R^{-1} - \frac{2}{m} \sigma \lambda \right] P_\mu + \frac{1}{m} \sigma^2 \lambda^2 \\
		& - C_{\ref*{cst-cutoff-2}}^{\frac{1}{2}} P_\mu^{\frac{3}{2}} R^{-\frac{1}{2}}  - 4 k_1 \mu P_\mu Q_\mu^{\frac{1}{2}} \\
		& + \left( 4 m  - 4 k_1 \mu  - 2 C_{\ref*{cst-cutoff-2}} \mu R^{-1} - 2 \mu C_{\ref*{cst-cutoff-2}}^{\frac{1}{2}} R^{-\frac{1}{2}} P_\mu^{\frac{1}{2}} \right) Q_\mu \numberthis \label{c-2}
	\end{align*}
	where $C_{\ref*{cst-cutoff-2}}$ only depends on $C_{\ref*{cst-cutoff}}$ and
	\begin{align*}
		\sigma = \sigma_R = \frac{\sup_{B_{R}} (|\nabla_b v|^2 + \mu_R |v_0|^2)}{\sup_{B_{2R}} (|\nabla_b v|^2 + \mu_R |v_0|^2)} \in [0,1].
	\end{align*}
	In particular, if $k_1 \mu \leq \varepsilon \leq \frac{1}{2 m} $, then 
	\begin{align*}
		0 \geq & \left( \frac{1}{m} -\varepsilon \right) P_\mu^2 - \left[ 2 k + 2 k_1 \mu + 16 \mu^{-1} + \varepsilon^{-1} C_{\ref*{cst-cutoff-2}} R^{-1} - \frac{2}{m} \sigma \lambda \right] P_\mu + \frac{1}{m} \sigma^2 \lambda^2 \\
		& + \left( 4 m - 12 \varepsilon  - 2 C_{\ref*{cst-cutoff-2}} \mu R^{-1} - 2 \mu C_{\ref*{cst-cutoff-2}}^{\frac{1}{2}} R^{-\frac{1}{2}} P_\mu^{\frac{1}{2}} \right) Q_\mu \numberthis \label{c-3}
	\end{align*}
\end{lemma}


\begin{proof}
	Using Lemma \ref{c-lem-bochner} and \eqref{c-1}, we can do the following calculation
	\begin{align} 
		\chi \Delta_b ( \chi \Phi) & \geq   \frac{1}{m} (\chi \lambda + \chi |\nabla_b v|^2)^2 + 2 \langle \nabla_b \chi, \nabla_b (\chi \Phi) \rangle - 2 \chi^{-1} |\nabla_b \chi|^2 (\chi \Phi) + (\chi \Phi) \Delta_b \chi \nonumber \\
		& - 2 \chi \langle \nabla_b (\chi \Phi), \nabla_b v \rangle + 2 \langle \nabla_b \chi, \nabla_b v \rangle (\chi \Phi) + 2 \mu \chi^2 |v_0|^2 \langle \nabla_b \chi, \nabla_b v \rangle - 4 k_1 \mu \chi^3 |\nabla_b v|^2 |v_0|  \nonumber \\
		& + \left(4 m - 4 k_1 \mu \chi - 4 \mu \chi^{-1} |\nabla_b \chi|^2 + \mu \Delta_b \chi \right) \chi^2 |v_0|^2 \nonumber \\
		& -  \left[ 2 k \chi + 2 k_1 \mu \chi^2 + 16 \mu^{-1} \right] \chi |\nabla_b v|^2  . \label{c-9}
	\end{align}	
	Note that
	\begin{align*}
		\sup_{B_R} \Phi \leq \chi (x_{\mu}) \Phi (x_{\mu}) \leq \chi (x_{\mu}) \sup_{B_{2R}} \Phi \leq \chi (x_{\mu}) \sup_{B_{2R}} (|\nabla_b v|^2 + \mu |v_0|^2).
	\end{align*}
	Hence $\sigma \leq \chi (x_\mu) \leq 1$.
	By the assumption that $x_\mu$ is a maximum point, 
	\begin{align*}
		\nabla_b (\chi \Phi) (x_\mu) =0, \Delta_b (\chi \Phi) (x_\mu) \leq 0.
	\end{align*}
	Then combing with \eqref{b-cutoff} and Cauchy inequality, the estimate \eqref{c-9} can be rewritten at $x_\mu$ as
	\begin{align*}
		0 \geq  & \frac{1}{m} (\sigma \lambda + P_\mu)^2 - C_{\ref*{cst-cutoff-2}} R^{-1} S_\mu - C_{\ref*{cst-cutoff-2}}^{\frac{1}{2}} R^{-\frac{1}{2}} P_\mu^{\frac{1}{2}} S_\mu - \mu C_{\ref*{cst-cutoff-2}}^{\frac{1}{2}} R^{-\frac{1}{2}} P_\mu^{\frac{1}{2}} Q_\mu - 4 k_1 \mu P_\mu Q_\mu^{\frac{1}{2}}  \nonumber \\
		& + \left(4 m  - 4 k_1 \mu  - C_{\ref*{cst-cutoff-2}} \mu R^{-1} \right) Q_\mu  -  \left[ 2 k  + 2 k_1 \mu  + 16 \mu^{-1} \right] P_\mu 
	\end{align*}
	which leads the conclusion \eqref{c-2}.

	The conclusion \eqref{c-3} is due to 
	\begin{align*}
		4 k_1 \mu P_\mu Q_\mu^{\frac{1}{2}} & \leq 4 \varepsilon P_\mu Q_\mu^{\frac{1}{2}} \leq \frac{\varepsilon}{2} P_\mu^2 + 8 \varepsilon Q_\mu \\
		C_{\ref*{cst-cutoff-2}}^{\frac{1}{2}} P_\mu^{\frac{3}{2}} R^{-\frac{1}{2}} &\leq \frac{\varepsilon}{2} P_\mu^2 + \frac{\varepsilon^{-1}}{2} C_{\ref*{cst-cutoff-2}} R^{-1} P_\mu
	\end{align*}
\end{proof}

Now we begin to estimate $\nabla_b v$.
Let 
\begin{align*}
	a =\max_{B_{2 R}} P = \max_{B_{2 R}} \chi |\nabla_b v|^2
\end{align*}
which is independent of $\mu$.

\begin{lemma} \label{c-lem-bound}
	Suppose $(M^{2m +1}, \theta)$ is a complete noncompact pseudo-Hermitian manifold with
	\begin{align*}
		R_* + 2 (m-2) \tau J \geq - k, \mbox{ and } |A|, |\mbox{div } A| \leq k_1
	\end{align*}
	for some $k, k_1 \geq 0$. If $u$ is a positive solution of 
	\begin{align}
		\Delta_b u = - \lambda u, \mbox{ on } B_{2 R}
	\end{align}
	for some $\lambda \geq 0$ and $R \geq 1$, then
	\begin{align}
		\max_{B_R} \frac{|\nabla_b u|}{u} & \leq C(k, k_1) (1 + R^{- \frac{1}{2}}) \\ 
		\max_{B_R} \frac{|u_0|}{u} & \leq C(k, k_1) (1 + R^{-\frac{1}{2}})
	\end{align}
	where the constant $C(k,k_1)$ only depends on $m, k$ and $k_1$.
\end{lemma}

\begin{remark}
	The reason of the assumption of $R \geq 1$ is that we are more interested in the global estimates. For the case $R < 1$, one could also get a similar lemma as Lemma \ref{c-lem-bound} via some modification of \eqref{b-cutoff}.
\end{remark}

\begin{proof}
	The conclusions are obvious if the maximum value of $\chi \Phi$ is zero. Hence we assume that it is nonzero, that is $S_\mu \neq 0$.
	Choose 
	\begin{align*}
		\varepsilon = \frac{1}{6 m}
	\end{align*}
	and
	\begin{align}
		\mu^{-1} = \mu_R^{-1} = ( k_1 + 1) \varepsilon^{-1} + 2 \left( C_{\ref*{cst-cutoff-2}} R^{-1} + C_{\ref*{cst-cutoff-2}}^{\frac{1}{2}} R^{-\frac{1}{2}} a^{\frac{1}{2}} \right) 
	\end{align}
	which makes
	\begin{align*}
		4 m - 12 \varepsilon  - 2 C_{\ref*{cst-cutoff-2}} \mu R^{-1} - 2 \mu C_{\ref*{cst-cutoff-2}}^{\frac{1}{2}} R^{-\frac{1}{2}} P_\mu^{\frac{1}{2}} \geq m.
	\end{align*}
	Hence \eqref{c-3} implies that 
	\constantnumber{cst-1}
	\begin{align}
		0 \geq & \frac{5}{6 m}  P_\mu^2 - C_{\ref*{cst-1}} \left[  k +k_1 + 1 + C_{\ref*{cst-cutoff-2}} R^{-1} +  C_{\ref*{cst-cutoff-2}}^{\frac{1}{2}} R^{-\frac{1}{2}} a^{\frac{1}{2}} \right] P_\mu  + m Q_\mu,
	\end{align}
	where $C_{\ref*{cst-1}}$ only depends on $m$.
	Then we find
	\begin{align}
		P_\mu & \leq \frac{6m}{5} C_{\ref*{cst-1}} \left[  k +k_1 + 1 + C_{\ref*{cst-cutoff-2}} R^{-1} +  C_{\ref*{cst-cutoff-2}}^{\frac{1}{2}} R^{-\frac{1}{2}} a^{\frac{1}{2}} \right] \\
		Q_\mu & \leq  \frac{3}{10} C_{\ref*{cst-1}}^2 \left[  k +k_1 + 1 + C_{\ref*{cst-cutoff-2}} R^{-1} +  C_{\ref*{cst-cutoff-2}}^{\frac{1}{2}} R^{-\frac{1}{2}} a^{\frac{1}{2}} \right]^2
	\end{align}
	One can easily check that 
	\begin{align*}
		\mu^{-1}  \bigg[  k +k_1 + 1 + C_{\ref*{cst-cutoff-2}} R^{-1} +  C_{\ref*{cst-cutoff-2}}^{\frac{1}{2}} R^{-\frac{1}{2}} a^{\frac{1}{2}} \bigg]  \leq 6m (k + k_1 + 1) + \frac{1}{2}
	\end{align*}
	which yields that 
	\begin{align*}
		a \leq S_\mu = P_\mu +  \mu Q_\mu \leq & \left[ \frac{6m}{5} C_{\ref*{cst-1}} + \frac{3}{10} C_{\ref*{cst-1}}^2 \left( 6m (k + k_1 + 1) + \frac{1}{2} \right) \right]\\
		& \times \left[  k +k_1 + 1 + C_{\ref*{cst-cutoff-2}} R^{-1} +  C_{\ref*{cst-cutoff-2}}^{\frac{1}{2}} R^{-\frac{1}{2}} a^{\frac{1}{2}} \right]
	\end{align*}
	Hence we have 
	\begin{align*}
		a^\frac{1}{2} \leq C(k, k_1) (1 + R^{-\frac{1}{2}}).
	\end{align*}
	where the constant $C(k,k_1)$ only depends on $m, k$ and $k_1$. The lemma will be obtained by 
	\begin{align*}
		\max_{B_R} |\nabla_b v|^2 & \leq a  \\
		\max_{B_R} |v_0|^2 & \leq \mu^{-1} (P_\mu + \mu Q_\mu).
	\end{align*}
\end{proof}

If $u$ is a global positive solution of $\Delta_b u = - \lambda u$, then we can get better estimate of its horizontal gradient and some gap of $\lambda$.

\begin{theorem} \label{c-thm-general}
	Suppose $(M^{2m +1}, \theta)$ is a complete noncompact pseudo-Hermitian manifold with
	\begin{align*}
		R_* + 2 (m-2) \tau J \geq - k, \mbox{ and } |A|, |\mbox{div } A| \leq k_1
	\end{align*}
	for some $k, k_1 \geq 0$. If $u$ is a global positive solution of 
	\begin{align}
		\Delta_b u = - \lambda u, \mbox{ on } M
	\end{align}
	for some $\lambda \geq 0$, then
	\begin{align}
		\lambda \leq \frac{1 - \sqrt{1 -  \varepsilon m}}{ \varepsilon} (k +  \varepsilon + 8 k_1 \varepsilon^{-1}) \label{c-4}
	\end{align}
	and 
	\begin{align}
		\frac{|\nabla_b u|^2}{u^2} \leq \frac{\zeta  - \frac{2}{m}  \lambda + \sqrt{\zeta^2 - \frac{4}{m} \zeta \lambda + \frac{4 \varepsilon}{m} \lambda^2 } }{2 (\frac{1}{m} - \varepsilon)}  + \frac{\zeta^2 - \frac{4}{m} \zeta \lambda + \frac{4 \varepsilon}{m} \lambda^2}{4 (1 - \varepsilon m)  \varepsilon } \label{c-10}
	\end{align}
	for any $\varepsilon \in (0, \frac{1}{6 m}) $ and
	\begin{align*}
		\zeta = 2 k + 18 \varepsilon  + 16 k_1  \varepsilon^{-1}.
	\end{align*}
\end{theorem}

\begin{proof}
	As above, we prove the theorem via the maximum principle of $\chi \Phi = S = P + \mu Q$ on $B_{2R}$. 
	Assume $u$ is non-constant. Otherwise the conclusion is trivial.
	Lemma \ref{c-lem-bound} shows that $|\nabla_b v|$ and $|v_0|$ are uniformly bounded on $M$, that is 
	\begin{align*}
		|\nabla_b v| \leq K, |v_0| \leq K
	\end{align*}
	for some $K >0 $ only depending on $k, k_1$. Let $\varepsilon \in (0, \frac{1}{6 m}), s>0$ and
	\begin{align}
		\mu^{-1} = \mu_R^{-1} = k_1 \varepsilon^{-1} + s + 2 \left( C_{\ref*{cst-cutoff-2}} R^{-1} + C_{\ref*{cst-cutoff-2}}^{\frac{1}{2}} R^{-\frac{1}{2}} K^{\frac{1}{2}} \right) \to k_1  \varepsilon^{-1} + s , \mbox{ as } R \to \infty \label{c-5}
	\end{align}
	such that
	\begin{align*}
		4 m - 12 \varepsilon  - 2 C_{\ref*{cst-cutoff-2}} \mu R^{-1} - 2 \mu C_{\ref*{cst-cutoff-2}}^{\frac{1}{2}} R^{-\frac{1}{2}} P_\mu^{\frac{1}{2}} \geq m.
	\end{align*}
	Moreover, 
	\begin{align*}
		\sigma = \sigma_R = & \frac{\sup_{B_{R}} (|\nabla_b v|^2 + \mu_R |v_0|^2)}{\sup_{B_{2R}} (|\nabla_b v|^2 + \mu_R |v_0|^2)} \\
		\geq & \frac{\sup_{B_{R}} (|\nabla_b v|^2 + (k_1 \varepsilon^{-1} + s)  |v_0|^2) - | \mu_R - (k_1 \varepsilon^{-1} + s) | \sup_{B_R} |v_0|^2 }{\sup_{B_{2R}} (|\nabla_b v|^2 + (k_1 \varepsilon^{-1} + s)  |v_0|^2) + |\mu_R - (k_1 \varepsilon^{-1} + s) | \sup_{B_{2R}} |v_0|^2} \\
		& \to 1 , \mbox{ as } R \to \infty.
	\end{align*}
	But $\sigma_R \leq 1$. So we have
	\begin{align}
		\lim_{R \to \infty} \sigma_R = 1.
	\end{align}
	At the maximum point $x_{\mu_R}$ of $\chi \Phi$, the estimate \eqref{c-3} implies that 
	\begin{align}
		0 \geq & - \frac{\left[ 2 k + 2 \varepsilon + 16 \mu_R^{-1} + \varepsilon^{-1} C_{\ref*{cst-cutoff-2}} R^{-1} - \frac{2}{m} \sigma_R \lambda \right]^2}{4 (\frac{1}{m} - \varepsilon)} + \frac{1}{m} \sigma_R^2 \lambda^2 
	\end{align}
	Taking $R \to \infty$, the result is
	\begin{align}
		0 \geq - \frac{1}{1 - \varepsilon m} \left(  \varepsilon \lambda^2 -  \zeta_s \lambda + \frac{m}{4} \zeta_s^2  \right)
	\end{align}
	where
	\begin{align*}
		\zeta_s = 2 k + 2 \varepsilon + 16 s + 16 k_1  \varepsilon^{-1}.
	\end{align*}
	Hence for any $s > 0$ and $\varepsilon \in (0, \frac{1}{6 m})$,
	\begin{align}
		\lambda \leq \frac{1- \sqrt{1 - \varepsilon m}}{ 2 \varepsilon} \zeta_s \quad \mbox{ or } \quad \lambda \geq \frac{1+ \sqrt{1 - \varepsilon m}}{2 \varepsilon} \zeta_s.
	\end{align}
	The second case is impossible since 
	\begin{align*}
		\frac{1+ \sqrt{1 - \varepsilon m}}{2 \varepsilon} \zeta_s \geq 8 s \varepsilon^{-1} \to \infty, \mbox{ as } \varepsilon \to 0. 
	\end{align*}
	By taking $s \to 0$, the conclusion \eqref{c-4} is obtained.

	For the global estimates of $|\nabla_b v|$ and $|v_0|$, Lemma \ref{c-lem-bochner} and the choice \eqref{c-5} show that 
	\begin{align}
		0 \geq  \left( \frac{1}{m} -\varepsilon \right) P_\mu^2 - \left[ \zeta_s + T_R - \frac{2}{m} \sigma_R \lambda \right] P_\mu + \frac{1}{m} \sigma_R^2 \lambda^2  + m Q_\mu \label{c-6}
	\end{align}
	where 
	\begin{align*}
		T_R = (32 + \varepsilon^{-1}) C_{\ref*{cst-cutoff-2}}^{-1} R^{-1} + 32 C_{\ref*{cst-cutoff-2}}^{\frac{1}{2}} R^{-\frac{1}{2}} K^{\frac{1}{2}} \to 0, \mbox{ as } R \to \infty.
	\end{align*}
	Hence \eqref{c-6} guarantees that 
	\begin{align}
		P_\mu \leq & \frac{\zeta_s + T_R - \frac{2}{m} \sigma_R \lambda + \sqrt{(\zeta_s + T_R - \frac{2}{m} \sigma_R \lambda)^2 - 4 \frac{1}{m} (\frac{1}{m} - \varepsilon) \sigma_R^2 \lambda^2} }{2 (\frac{1}{m} - \varepsilon)} \nonumber \\
		& \to \frac{\zeta_s  - \frac{2}{m}  \lambda + \sqrt{\zeta_s^2 - \frac{4}{m} \zeta_s \lambda + \frac{4 \varepsilon}{m} \lambda^2 } }{2 (\frac{1}{m} - \varepsilon)} \label{c-7}
	\end{align}
	and
	\begin{align}
		\mu_R Q_\mu \leq & \frac{\mu_R}{m} \frac{\left[ \zeta_s + T_R - \frac{2}{m} \sigma_R \lambda \right]^2 - 4 \frac{1}{m} (\frac{1}{m} - \varepsilon) \sigma_R^2 \lambda^2 }{4 (\frac{1}{m} - \varepsilon)} \nonumber \\
		& \to \frac{\zeta_s^2 - \frac{4}{m} \zeta_s \lambda + \frac{4 \varepsilon}{m} \lambda^2}{4 (1 - \varepsilon m)  (k_1 \varepsilon^{-1} + s) } 
	\end{align}
	as $R \to \infty$.
	The proof is finished by choosing $s = \varepsilon$ and
	\begin{align*}
		|\nabla_b v|^2 \leq \lim_{R \to \infty} (P_{\mu_R} + \mu_R Q_{\mu_R}).
	\end{align*}
\end{proof}

The Sasakian condition ($k_1 =0$) will simplify the conclusion in Theorem \ref{c-thm-general}.

\begin{corollary} \label{c-thm-sasakian}
	Suppose $(M^{2m +1}, \theta)$ is a complete noncompact Sasakian manifold with
	\begin{align*}
		R_* \geq -  k 
	\end{align*}
	for some $k \geq 0$. If $u$ is a global positive solution of 
	\begin{align}
		\Delta_b u = - \lambda u, \mbox{ on } M
	\end{align}
	for some $\lambda \geq 0$, then
	\begin{align}
		\lambda \leq \frac{m k}{2} \label{c-8}
	\end{align}
	and 
	\begin{align}
		\frac{|\nabla_b u|^2}{u^2} \leq \left( k^2 - \frac{2}{m} k \lambda \right) \varepsilon^{-1} +  \sqrt{ m^2 k^2 - 2k m \lambda} +  \left( \frac{18}{m} + 1 \right) (mk -\lambda) + \frac{ \lambda^2}{m} + O (\sqrt{\varepsilon}) \label{c-11}
	\end{align}
	where $ \frac{O (\sqrt{\varepsilon})}{\sqrt{\varepsilon}} \leq C (m, k) $ depending on $m$ and $k$ as $\varepsilon \to 0$.
\end{corollary}

\begin{proof}
	The estimate \eqref{c-8} follows from \eqref{c-4} by taking $\varepsilon \to 0$.

	Since $\zeta = 2 k + 18 \varepsilon $, we have 
	\begin{align*}
		\zeta^2 - \frac{4}{m} \zeta \lambda + \frac{4 \varepsilon}{m} \lambda^2 = 4 \left( k^2 - \frac{2}{m} k \lambda \right) + 4 \left( 18 k - \frac{18 \lambda}{m} + \frac{ \lambda^2}{m} \right) \varepsilon + 18^2 \varepsilon^2.
	\end{align*}
	Hence 
	\begin{align*}
		\frac{\zeta  - \frac{2}{m}  \lambda + \sqrt{\zeta^2 - \frac{4}{m} \zeta \lambda + \frac{4 \varepsilon}{m} \lambda^2 } }{2 (\frac{1}{m} - \varepsilon)} = mk - \lambda +  \sqrt{ m^2 k^2 - 2 k m \lambda} + O (\sqrt{\varepsilon})
	\end{align*}
	and 
	\begin{align*}
		\frac{\zeta^2 - \frac{4}{m} \zeta \lambda + \frac{4 \varepsilon}{m} \lambda^2}{4 (1 - \varepsilon m)  \varepsilon } = \left( k^2 - \frac{2}{m} k \lambda \right) \varepsilon^{-1} +  18 k - \frac{18 \lambda}{m} + \frac{ \lambda^2}{m}  + O (\varepsilon)
	\end{align*}
	where $ \frac{O (\sqrt{\varepsilon})}{\sqrt{\varepsilon}},  \frac{O (\varepsilon)}{\varepsilon} \leq C (m ,k) $ only depending on $m$ and $k$ as $\varepsilon \to 0$. Hence \eqref{c-11} follows from \eqref{c-10}.
\end{proof}

The pseudo-harmonic function is the solution of $\Delta_b u =0$. We have the following Liouville theorem of positive pseudo-harmonic functions.

\begin{theorem} \label{c-thm-liouville}
	Suppose that $(M^{2m+1}, \theta)$ is a complete noncompact Sasakian manifold with 
	\begin{align*}
		R_* \geq -k
	\end{align*}
	for some $k \geq 0$.
	Then any positive pseudo-harmonic function $u$ on $M$ satisfies
	\begin{align}
		\frac{|\nabla_b u|^2}{u^2} \leq  \left(m + 16 + 4 \sqrt{ 2m + 16} \right) k. \label{c-12}
	\end{align}
	If the pseudo-Hermitian Ricci curvature of $M$ is nonnegative, that is $k =0$,
	then any positive pseudo-harmonic function on $M$ must be constant.
\end{theorem}

\begin{proof}
	It suffices to prove \eqref{c-12}.
	By the proof of Theorem \ref{c-thm-general}, we have
	\begin{align*}
		|\nabla_b v|^2 \leq  \lim_{R \to \infty} (P_{\mu_R} + \mu_R Q_{\mu_R}) 
		\leq  \frac{\zeta_s}{2(\frac{1}{m} - \varepsilon)} + \frac{\zeta_s^2}{4 (1- \varepsilon m) s}  
		\to (m + 16) k + 8 (m + 8) s + \frac{ k^2}{s}
	\end{align*}
	as $\varepsilon \to 0$. Now if $k \neq 0$, then we choose 
	\begin{align*}
		s = \frac{k}{\sqrt{ 8 m + 64}}
	\end{align*}
	which implies \eqref{c-12}. If $k =0$, then by taking $s \to 0$ and \eqref{c-12} is obtained.
\end{proof}

\begin{remark}
	One can also prove the Liouville part of Theorem \ref{c-thm-liouville} directly by the maximum principle as Lemma \ref{c-lem-bound} with choosing
	\begin{align*}
		\mu^{-1} = \mu_R^{-1} = 4 \left( C_{\ref*{cst-cutoff-2}} R^{-1} + C_{\ref*{cst-cutoff-2}}^{\frac{1}{2}} R^{-\frac{1}{2}} a^{\frac{1}{2}} \right)
	\end{align*}
	where 
	\begin{align*}
		a = \max_{B_{2 R}} \chi |\nabla_b v|^2.
	\end{align*}
\end{remark}

Similarly, Theorem \ref{c-thm-general} also gives the horizontal gradient estimates of positive pseudo-harmonic maps on pseudo-Hermitian manifolds with nonnegative pseudo-Hermitian Ricci curvature.

\constantnumber{cst-2}
\begin{theorem}
	Suppose that $(M^{2m+1}, \theta)$ is a complete noncompact pseudo-Hermitian manifold with nonnegative pseudo-Hermitian Ricci curvature and  
	\begin{align*}
		|A|, | \mbox{div } A | \leq k_1
	\end{align*}
	for some $k_1 \geq 0$.
	Then any positive pseudo-harmonic function $u$ on $M$ satisfies
	\begin{align}
		\frac{|\nabla_b u|^2}{u^2} \leq C_{\ref*{cst-2}} \left( k_1^2 \varepsilon^{-3} + k_1 \varepsilon^{-1} + \varepsilon \right)
	\end{align}
	 where $\varepsilon \in (0, \frac{1}{6m})$ and $C_{\ref*{cst-2}}$ only depends on $m$.
\end{theorem}

\bibliographystyle{plain}

\bibliography{finalref}

\vspace{12 pt}

Yibin Ren

\emph{College of Mathematics, Physics and Information Engineering}

\emph{Zhejiang Normal University}

\emph{Jinhua, 321004, Zhejiang, P.R. China}

allenryb@outlook.com

\end{document}